\documentclass[11pt]{article}

\usepackage{amssymb,amsmath,amsfonts,amsthm,mathrsfs}
\usepackage{enumerate}
\usepackage[colorlinks=true, pdfstartview=FitV, linkcolor=blue, citecolor=blue, urlcolor=blue]{hyperref}
\usepackage{graphicx,color}
\usepackage{cite}
\usepackage{indentfirst}
\allowdisplaybreaks[4]

\begin{document}
\setlength{\parindent}{2em}

\newtheorem{thm}{Theorem}
\newtheorem*{thmA}{Theorem A}

\newtheorem{lem}{Lemma}
\newtheorem*{lemA}{Lemma A}  \newtheorem*{lemB}{Lemma B}

\newtheorem*{cor}{Corollary}

\theoremstyle{definition}
\newtheorem{definition}{Definition}

\newtheorem{remark}{Remark}
\numberwithin{remark}{section}
\numberwithin{thm}{section}
\numberwithin{lem}{section}
\numberwithin{definition}{section}
\numberwithin{equation}{section}

\def\rn{{\mathbb R^n}}  \def\sn{{\mathbb S^{n-1}}}

\title{\bf\Large LIMITING WEAK-TYPE BEHAVIORS FOR SINGULAR INTEGRALS WITH ROUGH $L\log L(\sn)$ KERNELS
\footnotetext{{\it Key words and phrases}. Singular integral operator, bilinear singular integral operator, limiting weak-type behaviors, $L\log L(\sn)$ rough kernel.
\newline\indent\hspace{1mm} {\it 2020 Mathematics Subject Classification}. Primary 42B20; Secondary 42B25.
\newline\indent\hspace{1mm} The authors were supported partly by NSFC (Nos. 11771358, 11871101).}}

\date{}

\author{Moyan Qin, Huoxiong Wu, Qingying Xue\footnote{Corresponding
author, E-mail: \texttt{qyxue@bnu.edu.cn}}}

\maketitle

\begin{center}
	\begin{minipage}{13cm}
	{\small {\bf Abstract}\quad
	Let \(\Omega\) be a function of homogeneous of degree zero and vanish on the unit sphere \(\sn\). In this paper, we investigate the limiting weak-type behavior for singular integral operator $T_\Omega$ associated with rough kernel \(\Omega\). We show that, if \(\Omega\in L\log L(\sn)\), then
	\[\lim_{\lambda\to0^+}\lambda|\{x\in\rn:|T_\Omega(f)(x)|>\lambda\}| = n^{-1}\|\Omega\|_{L^1(\sn)}\|f\|_{L^1(\rn)},\quad0\le f\in L^1(\rn).\]
	Moreover, \(n^{-1}\|\Omega\|_{L^1(\sn)}\) is a lower bound of weak-type norm of $T_\Omega$ when \(\Omega\in L\log L(\sn)\). Corresponding results for rough bilinear singular integral operators defined in the form \(T_{\vec\Omega}(f_1,f_2) = T_{\Omega_1}(f_1)\cdot T_{\Omega_2}(f_2)\) have also been established.}
\end{minipage}
\end{center}

\section{Introduction}\label{s1}


It is well-known that the existence and boundedness of singular integral operator were first investigated by Calder\'on and Zygmund \cite{CZ1952,CZ1956}. Let \(\Omega\in L^1(\sn)\) and satisfy the following vanishing condition:
\begin{equation}\label{eq1.1}
	\int_{\sn}\Omega(\theta)d\sigma(\theta)=0,
\end{equation}
where $\sn$ denotes the unit sphere on the Euclidean space $\rn$, $d\sigma(\cdot)$ is the induced Lebesgue measure on $\sn$. It was shown in \cite{CZ1956} that the rough singular integral operator \(T_\Omega\) defined by
\[T_\Omega(f)(x)=\text{p.v.}\int_\rn\frac{\Omega(y/|y|)}{|y|^n}f(x-y)dy\]
is bounded on \(L^p(\rn)\) for \(1<p<\infty\), provided that $\Omega\in L\log L(\sn)$, where
$$ L\log L(\sn):=\{\Omega\in L^1(\sn):\,
\|\Omega\|_{L\log L(\sn)}:=\int_{\sn}|\Omega(\theta)|\log(e+|\Omega(\theta)|)d\sigma(\theta)<\infty\}.$$
Moreover, if \(\Omega\) is an odd function on \(\sn\), \(\Omega\in L^1(\sn)\) is sufficient to guarantee the \(L^p\)-boundedness of \(T_\Omega\).

In 1979, Connett \cite{C1979}, Ricci and Weiss \cite{RW1979} independently proved that \(T_\Omega\) is still $L^p$ bounded when \(\Omega\) belongs to a larger space \(H^1(\sn)\). Here \(H^1(\sn)\) denotes the Hardy space on \(\sn\) which contains \(L\log L(\sn)\) as a proper subspace. The result in \cite{C1979,RW1979} can be formulated in the following way.

\begin{thmA}[\cite{C1979,RW1979}]
Let \(\Omega\) be an integrable function on \(\sn\) with mean value zero which satisfies \(\Omega\in H^1(\sn)\). Then \(T_\Omega\) can be extended to be a bounded operator from \(L^p(\rn)\) into itself for \(1<p<\infty\). More precisely, there exists a positive constant $C$, independent of $\Omega$ and $f$, such that
\begin{equation}\label{eq1.2}
	\|T_\Omega(f)\|_{L^p(\rn)} \le C\|\Omega\|_{H^1(\sn)}\|f\|_{L^p(\rn)}.
\end{equation}
\end{thmA}

\begin{remark}
The precise constant on the right side of inequality (\ref{eq1.2}) was not given explicitly in \cite{C1979} and \cite{RW1979}. However, one can prove inequality (\ref{eq1.2}) by following the main steps in \cite{GS1999,RW1979}. On the other hand, inequality (\ref{eq1.2}) is a special case of Theorem 1 of \cite{LMW2016}, where the singular integral operators related to homogeneous mappings were studied.
\end{remark}

Since then, great achievements have been made in this area. Among them are the celebrated works of Christ \cite{C1988}, Christ and Rubio de Francia \cite{CR1988}, Hofmann \cite{H1988}, Seeger \cite{S1996}, Tao \cite{T1999}, etc. In \cite{C1988}, Christ established the weak type \((1,1)\) boundedness of Bochner-Riesz means and rough Hardy-Littlewood maximal operator. The weak type $(1,1)$ boundedness of $T_\Omega$ was given by Christ and Rubio de Francia \cite{CR1988} when \(\Omega\in L\log L(\sn)\) for \(n\le5\), and independently obtained by Hofmann \cite{H1988} when \(\Omega\in L^q(\mathbb S^1)\) for \(n=2\). It was eventually proved by Seeger \cite{S1996} and Tao \cite{T1999} that \(T_\Omega\) is of weak type \((1,1)\) when \(\Omega\in L\log L(\sn)\) for all dimensions. More related works, in particular, weighted version of the above results, can be found in references \cite{D1993,DR1986,HRT2017,LPRR2019,V1996,W1990}.


Here we will focus on the best constants problem of weak endpoints estimates for some important operators in Harmonic analysis. This belongs to less fine problems and has attracted lots of attentions. For example, for $n=1$, Davis \cite{D1974} obtained the best constant of weak-type $(1,1)$ for Hilbert transform, and Melas \cite{M2003} proved that $\|M\|_{L^1\to L^{1,\infty}}=\frac{11+\sqrt{61}}{11}$ for Hardy-Littlewood maximal operator \(M\). However, for $n\ge2$, things become more subtle. Even for the well-known Riesz transform, there is no such information.

In 2004, Janakiraman \cite{J2004} considered the Riesz transform and the singular integral operator $T_\Omega$. It was shown that the constant $\|T_\Omega\|_{L^1\to L^{1,\infty}}$ is at worst $C(\log n\|\Omega\|_{L^1}+\|T_\Omega\|_{L^2\to L^2}^2)$, if $\Omega$ satisfies (\ref{eq1.1}) and the new regularity condition that
\begin{equation}\label{eq1.3}
	\sup_{\xi\in\sn}\int_\sn|\Omega(\theta)-\Omega(\theta+\delta\xi)| d\sigma(\theta) \le Cn\delta\int_\sn|\Omega(\theta)|d\sigma(\theta),\quad\text{for }0<\delta<\frac1n.
\end{equation}
To explore the lower bounds of $\|T_\Omega\|_{L^1\to L^{1,\infty}}$, in 2006, under the same kernel conditions, Janakiraman \cite{J2006} gave the following limiting weak-type behavior for $T_\Omega$:
$$\lim_{\lambda\to0^+}\lambda|\{x\in\rn:|T_\Omega(f)(x)|>\lambda\}|= \frac{\|\Omega\|_{L^1(\sn)}}n\|f\|_{L^1(\rn)},\quad\text{for }0\le f\in L^1(\rn).$$
Similar result have also been obtained for Hardy-Littlewood maximal function $M$.

Subsequently, Janakiraman's results were essentially improved by Ding and Lai \cite{DL2017} in the way that the kernel condition (\ref{eq1.3}) was replaced by
a weaker $L^1$-Dini condition, that is, $\Omega\in L^1(\sn)$ and satisfies
$$\int_0^1\frac{\omega_1(\delta)}\delta d\delta<\infty,$$
where $\omega_1(\delta)=\sup_{|h|\le\delta}\int_\sn|\Omega(\theta+h)-\Omega(\theta)|d\sigma(\theta)$ for $\delta>0$.
They also proved the similar results for maximal type operators with Dini type kernels, see \cite{DL2017.}.


As we mentioned before \cite{S1996,T1999}, \(\Omega\in L\log L(\sn)\) is a sufficient condition to guarantee the weak-type $(1,1)$ boundedness of $T_\Omega$. However, both the results of Janakiraman \cite{J2006}, Ding and Lai \cite{DL2017} were obtained under certain smoothness conditions. Therefore, one may ask, whether these smoothness conditions can be removed. Note that, the following space inclusion relationships are true:
\[L^1-\text{Dini}\subsetneq L\log L(\sn) \subsetneq H^1(\sn)\subsetneq L^1(\sn).\]
Thus it is quite natural to ask the following question:

\noindent{\bf Questions 1:} Does $T_\Omega$ still enjoy the limiting weak-type behaviors when \(\Omega\in L\log L(\sn)\)?


In this paper, we give a positive answer to this question. Our first main result is as follows:

\begin{thm}\label{thm1}
Let $\Omega\in L\log L(\sn)$ and satisfy the vanishing condition \((\ref{eq1.1})\). Then for all $0\le f\in L^1(\rn)$, it holds that
\begin{enumerate}[\rm{(i)}]
	\item $\displaystyle\lim_{\lambda\to0^+}\lambda|\{x\in\rn:|T_\Omega(f)(x)|>\lambda\}| = \frac1n\|\Omega\|_{L^1(\sn)}\|f\|_{L^1(\rn)};$
\item[\rm{(ii)}] $\displaystyle\lim_{\lambda\to0^+}\lambda\big|\big\{x\in\rn:\big||T_\Omega(f)(x)| - \|f\|_{L^1(\rn)}|\Omega(x)||x|^{-n}\big|>\lambda\big\}\big|=0.$
\end{enumerate}
\end{thm}


Before stating our second main result, let's recall some background for multilinear operators. It was Coifman and Grafakos \cite{CG1992,CG1992.} who first studied the bilinear singular integral operators given by inner products of finite vectors of Calder\'{o}n-Zygmund operators as follows:
$$T_{\vec\Omega}(f_1,f_2)(x) = T_{\Omega_1}(f_1)(x)T_{\Omega_2}(f_2)(x).$$
It was originated from investigating the boundedness of the determinant of the Jacobian \cite{CLMS1993}. They \cite{CG1992,CG1992.} proved that \(T_{\vec\Omega}\) is bounded from the product spaces of the Hardy spaces to the Hardy space \(H^r(\rn)\) and the weak Hardy space \(H^{r,\infty}(\rn)\), and from the product spaces of the Lebesgue spaces to the Hardy space \(H^r(\rn)\) and the weak Hardy space \(H^{r,\infty}(\rn)\) for different \(p\) and \(q\) when \(\Omega\) satisfies a certain number of moments vanishing and a strong smoothness condition.

Inspired by these works, Ding and Lu \cite{DL1999} essentially improved the result in \cite{CG1992.} by replacing their smoothness kernel condition to a weaker  \(L^s\)-Dini condition for \(s>1\). Later on, Ding and Lu \cite{DL2003} extended the result in \cite{DL1999} to multilinear fractional integral operators.


On the other hand, when \(\Omega_1,\Omega_2\in L\log L(\sn)\), it is easy to see that the \(L^1\times L^1\to L^{1/2,\infty}\) boundedness of \(T_{\vec\Omega}\) holds (see Lemma \ref{lem23}). Therefore the general question along these lines is the following:

\noindent{\bf Question 2:} What kinds of limiting weak-type behaviors does $T_{\vec\Omega}$ enjoy when \(\Omega_1,\Omega_2\in L\log L(\sn)\)?


This question will be addressed by our next theorem.

\begin{thm}\label{thm2}
Let $\Omega_1,\Omega_2\in L\log L(\sn)$ and both of them satisfy the vanishing condition \((\ref{eq1.1})\). Then for all $0\le f_1,f_2\in L^1(\rn)$, we have
\begin{enumerate}[\rm{(i)}]
	\item $\displaystyle\lim_{\lambda\to0^+}\lambda^{1/2}|\{x\in\rn:|T_{\vec\Omega}(f_1,f_2)(x)|>\lambda\}| = \frac1n\|\Omega_1\Omega_2\|_{L^{1/2}(\sn)}^{1/2}\prod\limits_{i=1}^2\|f_i\|_{L^1(\rn)}^{1/2};$
\item[\rm{(ii)}] $\displaystyle\lim_{\lambda\to0^+}\lambda^{1/2}\big|\big\{x\in\rn:\big||T_{\vec\Omega}(f_1,f_2)(x)| - \prod\limits_{i=1}^2\big(\|f_i\|_{L^1(\rn)}|\Omega_i(x)|\big)|x|^{-2n}\big|>\lambda\big\}\big|=0.$
\end{enumerate}
\end{thm}


An immediate application of Theorem \ref{thm1} and Theorem \ref{thm2} yields the lower bounds of weak norms of $T_\Omega$ and $T_{\vec\Omega}$ as follows.

\begin{cor}\label{cor}
Let $\Omega,\Omega_1,\Omega_2\in L\log L(\sn)$ and all of them satisfy the vanishing condition \((\ref{eq1.1})\). Then, it holds that
	\[\|T_\Omega\|_{L^1\to L^{1,\infty}} \ge \frac1n\|\Omega\|_{L^1(\sn)}\quad\text{and}\quad\|T_{\vec\Omega}\|_{L^1\times L^1\to L^{1/2,\infty}} \ge \frac1n{\|\Omega_1\Omega_2\|_{L^{1/2}(\sn)}^{1/2}}.\]
\end{cor}


\begin{remark}
For simplicity, we only consider the bilinear case. The same reasoning as in this paper shows that Theorem \ref{thm2} and this corollary can be extended easily to \(m\)-linear operators  for $m\ge 2.$
\end{remark}

One of our main ingredients is to give the fine bounds of $\|T_\Omega\|_{L^1(\mathbb{R}^n)\to L^{1,\infty}(\mathbb{R}^n)}$ for $\Omega\in L\log L(\mathbb{S}^{n-1})$ (see Lemma 3.2 below), which is the key of our proofs. To demonstrate our novelty, 
we now explain a little bit of the main ideas in the proof of Theorem \ref{thm1}. We first split the the function \(f\) into two terms: \(f_1\) has a compact support and \(f_2\) has a small norm. To prove Theorem \ref{thm1}, the first difficulty lies in dealing with the rough \(L\log L\) kernel. We need to decompose \(\Omega\) into \(\widetilde\Omega_\varepsilon\) and \(\Omega-\widetilde\Omega_\varepsilon\): one has the property of continuity on $\sn$ and another has a small norm. Unfortunately, this decomposition may break the vanishing property, therefore we reconstruct the kernels into \(\Omega_\varepsilon\) and \(\Omega-\Omega_\varepsilon\) to make sure that both of them enjoy the vanishing condition. Continue to the proof of Theorem \ref{thm1}, we need to focus on the contributions of \(T_{\Omega-\Omega_\varepsilon}(f_1)\) and \(T_{\Omega_\varepsilon}(f_1)\). For \(T_{\Omega-\Omega_\varepsilon}(f_1)\), the difficulty lies in that we need to prove that the weak \((1,1)\) norm of $T_{\Omega-\Omega_\varepsilon}$ tends to $0$ if \(\|\Omega-\Omega_\varepsilon\|_{L\log L(\sn)}\) tends to $0$. We note that, Seeger and Tao \cite{S1996,T1999} showed that, the weak \((1,1)\) norm of $T_{\Omega-\Omega_\varepsilon}$ can be controlled by \(C(\|\Omega-\Omega_\varepsilon\|_{L\log L(\sn)}+1)\), this does not fit our needs. Fortunately, by the recent work of Ding and Lai \cite{DL2019}, together with (\ref{eq1.2}), we can make the two ends meet. As for \(T_{\Omega_\varepsilon}(f_1)\), we take a viewpoint far away from origin. Then the contribution of \(T_{\Omega_\varepsilon}(f_1)\) is clear when we obtain an upper and a lower control of it. It is worth pointing out that we use a delicate localization method when dealing with the lower control of \(T_{\Omega_\varepsilon}(f_1)\).

The organization of this paper is as follows. In Section \ref{s2}, we will present some notations and basic lemmas, which will be used later. Section \ref{s3} will be devoted to the proof of Theorem \ref{thm1}, and the proof of Theorem \ref{thm2} will be given in Section \ref{s4}. Throughout this paper, the letter $C$ will stand for positive constants, not necessarily the same one at each occurrence, but always independent of $\Omega$ and $f$.

\section{Preliminaries}\label{s2}




In this section, we will recall and establish some auxiliary lemmas. First, we present the following lemma, which will play a fundamental role in our arguments.

\begin{lem}\label{lem21}
Let $\lambda,\gamma>0$, $\Phi\in L^{n/r}(\sn)$, $S\subset\sn$ be a measurable set. Then we have
$$\bigg|\bigg\{ x\in\rn:\frac{|\Phi(\frac x{|x|})|}{|x|^\gamma}>\lambda,\frac{x}{|x|}\in S\bigg\}\bigg|=\frac{\|\Phi\|_{L^{n/\gamma}(S)}^{n/\gamma}}{n\lambda^{n/\gamma}}.$$
\end{lem}

\begin{proof}
Note that
$$\bigcup_{\theta\in S}\bigg\{ r\theta:0<r<\left|\frac{\Phi(\theta)}{\lambda}\right|^{\frac1\gamma}\bigg\} = \bigg\{ x\in\rn:\frac{|\Phi(\frac x{|x|})|}{|x|^\gamma}>\lambda,\frac{x}{|x|}\in S\bigg\},$$
therefore
$$\bigg|\bigg\{ x\in\rn:\frac{|\Phi(\frac x{|x|})|}{|x|^\gamma}>\lambda,\frac{x}{|x|}\in S\bigg\}\bigg| = \int_S\int_0^{\big|\frac{\Phi(\theta)}{\lambda}\big|^{\frac1\gamma}}r^{n-1}drd\sigma(\theta) = \frac{\|\Phi\|_{L^{n/\gamma}(S)}^{n/\gamma}}{n\lambda^{n/\gamma}}.$$
\end{proof}

Before stating the next lemmas, we need to introduce a more general singular integral operator defined by
\[T_\Omega^K(f)(x) = \text{p.v.}\int_\rn\Omega(x-y)K(x,y)f(y)dy,\]
where \(\Omega\) is homogeneous of degree \(0\) and satisfies the vanishing condition (\ref{eq1.1}), \(K\) satisfies the size condition
\[|K(x,y)| \le \frac C{|x-y|^n},\]
and the regularity condition that for a fixed \(\delta\in(0,1]\),
\begin{align*}
	|K(x_1,y) - K(x_2,y)| \le C\frac{|x_1-x_2|^\delta}{|x_1-y|^{n+\delta}},\quad|x_1-y|>2|x_1-x_2|, \\
	|K(x,y_1) - K(x,y_2)| \le C\frac{|y_1-y_2|^\delta}{|x-y_1|^{n+\delta}},\quad|x-y_1|>2|y_1-y_2|.
\end{align*}

The following weak type \((1,1)\) bound criterion for \(T_\Omega^K\) was given by Ding and Lai.

\begin{lem}[\cite{DL2019}]\label{lem22}
Suppose \(\Omega\in L\log L(\sn)\). In addition, suppose \(\Omega\) and \(K\) satisfy some appropriate cancellation conditions such that \(T_\Omega^K\) is well defined for \(f\in C_c^\infty(\rn)\) and extends to a bounded operator on \(L^2(\rn)\) with bound \(C\|\Omega\|_{L\log L}\). Then for any \(\lambda>0\), we have
\[\lambda|\{x\in\rn:|T_\Omega^K(f)(x)|>\lambda\} \lesssim \mathcal C_\Omega\|f\|_{L^1},\]
where
$$\mathcal C_\Omega = \|\Omega\|_{L\log L(\sn)} + \int_\sn|\Omega(\theta)|\left(1+\log^+\frac{|\Omega(\theta)|}{\|\Omega\|_{L^1(\sn)}}\right)d\sigma(\theta).$$
\end{lem}

With Theorem A in hand, together with Lemma \ref{lem22}, we have the following key lemma.

\begin{lem}\label{lem23}
Let $\Omega,\Omega_1,\Omega_2\in L\log L(\sn)$ and all of them satisfy the vanishing condition $(\ref{eq1.1})$. Then $T_\Omega$ is bounded from $L^1(\rn)$ to $L^{1,\infty}(\rn)$ and $T_{\vec\Omega}$ is bounded from $L^1(\rn)\times L^1(\rn)$ to $L^{1/2,\infty}(\rn)$. They enjoy the following norm inequalities:
$$\|T_\Omega\|_{L^1(\rn)\to L^{1,\infty}(\rn)} \le C\mathcal C_\Omega \le C(\|\Omega\|_{L\log L(\sn)}+1),$$ and
$$\|T_{\vec\Omega}\|_{L^1(\rn)\times L^1(\rn)\to L^{1/2,\infty}(\rn)}\le C\prod_{i=1}^2\mathcal C_{\Omega_i} \le C\prod_{i=1}^2(\|\Omega_i\|_{L\log L(\sn)}+1).$$
\end{lem}

\begin{proof}
Note that $H^1(\sn)$ is the Hardy space on the unit sphere and contains $L\log L(\sn)$ as a proper subspace. It is natrual to ask whether \(\|\Omega\|_{H^1(\sn)} \le C\|\Omega\|_{L\log L(\sn)}\) holds or not. We have no answer at present time. However, we can show that
\begin{equation}\label{eq2.1}
	\|\Omega\|_{H^1(\sn)} \le C\mathcal C_\Omega.
\end{equation}

To prove inequality (\ref{eq2.1}), without losing  generality, we may assume that \(\|\Omega\|_{L^1(\sn)}>0\), otherwise nothing needs to be proved. Let \(a_0=\frac1{\sigma(\sn)}\|\Omega\|_{L^1(\sn)}\). For any \(\theta\in\sn\), the Hardy-Littlewood maximal function on the unit sphere is defined by
\[M(\Omega)(\theta) = \sup_{r>0}\frac1{\sigma(\widetilde B(\theta,r))}\int_{\widetilde B(\theta,r)}|\Omega(\theta')|d\sigma(\theta'),\]
where \(\widetilde B(\theta,r) = \{\theta\in\sn:\sigma(\theta,\theta')<r\}\). Then for \(\theta\in\sn\), it holds that
\[M(\Omega)(\theta) \ge \frac1{\sigma(\sn)}\int_\sn|\Omega(\theta')|d\sigma(\theta') = a_0.\]

It was known from \cite{C1982} that \(\|\Omega\|_{H^1(\sn)} \le \|M(\Omega)\|_{L^1(\sn)}\). Therefore, to prove inequality (\ref{eq2.1}), it is enough to show that
\begin{equation}\label{eq2.2}
	\|M(\Omega)\|_{L^1(\sn)} \le C\mathcal C_\Omega.
\end{equation}
To see inequality (\ref{eq2.2}) is true, we first notice that
\begin{align*}
	\|M(\Omega)\|_{L^1(\sn)} =& \int_0^\infty\sigma(\{\theta\in\sn:M(\Omega)(\theta)>\alpha\})d\alpha \\
	\le& \int_0^\infty\sigma(\{\theta\in\sn:M(\Omega\chi_{\{|\Omega|>\alpha/2\}})(\theta)>\frac\alpha2\})d\alpha \\
	&+ \int_0^\infty\sigma(\{\theta\in\sn:M(\Omega\chi_{\{|\Omega|\le\alpha/2\}})(\theta)>\frac\alpha2\})d\alpha \\
	\le& \int_{2a_0}^\infty\sigma(\{\theta\in\sn:M(\Omega\chi_{\{|\Omega|>\alpha/2\}})(\theta)>\frac\alpha2\})d\alpha + \int_0^{2a_0}\sigma(\sn)d\alpha \\
	\le& \int_{2a_0}^\infty\frac C{\alpha}\int_{\{|\Omega|>\alpha/2\}}|\Omega(\theta)|d\sigma(\theta)d\alpha + 2\|\Omega\|_{L^1(\sn)},
\end{align*}
where the last inequality follows from the weak type \((1,1)\) boundedness of Hardy-Littlewood maximal function on the unit sphere give by Knopf \cite{K1986,K1987} and Li \cite{L2013}.

Therefore
\begin{align*}
	\|M(\Omega)\|_{L^1(\sn)} \le& C\sum_{n=0}^\infty\int_{a_02^n}^{a_02^{n+1}}\frac1\alpha\int_{\{|\Omega|>\alpha\}}|\Omega(\theta)|d\sigma(\theta)d\alpha + 2\|\Omega\|_{L^1(\sn)} \\
	\le& C\sum_{n=0}^\infty\int_{\{|\Omega|>a_02^n\}}|\Omega(\theta)|d\sigma(\theta) + 2\|\Omega\|_{L^1(\sn)} \\
	=& C\sum_{n=0}^\infty(n+1)\int_{\{a_02^n<|\Omega|<a_02^{n+1}\}}|\Omega(\theta)|d\sigma(\theta) + 2\|\Omega\|_{L^1(\sn)} \\
	\le& C\sum_{n=0}^\infty\int_{\{a_02^n<|\Omega|<a_02^{n+1}\}}|\Omega(\theta)|\log^+\frac{|\Omega(\theta)|}{a_0}d\sigma(\theta) + (C+2)\|\Omega\|_{L^1(\sn)} \\
	\le& C\left(\sum_{n=0}^\infty\int_{\{a_02^n<|\Omega|<a_02^{n+1}\}}|\Omega(\theta)|\log^+\frac{|\Omega(\theta)|}{\|\Omega\|_{L^1(\sn)}}d\sigma(\theta) + \|\Omega\|_{L^1(\sn)}\right) \\
	\le& C\mathcal C_\Omega.
\end{align*}
Then we prove inequality (\ref{eq2.2}), and thus inequality (\ref{eq2.1}) is true.

Combining inequality (\ref{eq2.1}) with inequality (\ref{eq1.2}), we get
\begin{equation}\label{eq2.3}
	\|T_\Omega\|_{L^2(\rn)\to L^2(\rn)} \le C\|\Omega\|_{H^1(\sn)} \le C\mathcal C_\Omega.
\end{equation}

By Lemma \ref{lem22}, the condition \(\|T_\Omega\|_{L^2(\rn)\to L^2(\rn)}\le C\|\Omega\|_{L\log L(\sn)}\) is a sufficient condition to obtain the weak type \((1,1)\) bound criterion of \(T_\Omega\). However, we observe that, Lemma \ref{lem22} still holds when this sufficient condition is relaxed to \(\|T_\Omega\|_{L^2(\rn)\to L^2(\rn)}\le C\mathcal C_\Omega\). The main reason is that it was only used to estimate the good function \(g\), which comes from the Calder\'on-Zygmund decomposition of \(f\). Hence, inequality (\ref{eq2.3}) and Lemma \ref{lem22} yield that
$$\|T_\Omega\|_{L^1(\rn)\to L^{1,\infty}(\rn)} \le C\mathcal C_\Omega.$$

By the definition of \(\mathcal C_\Omega\), it is obvious that
\begin{align*}
	\mathcal C_\Omega \le& \|\Omega\|_{L\log L(\sn)} + \int_\sn|\Omega(\theta)|d\sigma(\theta) + \int_\sn|\Omega(\theta)|\log^+|\Omega(\theta)|d\sigma(\theta) \\
	&+ \|\Omega\|_{L^1(\sn)}\log^+\frac1{\|\Omega\|_{L^1(\sn)}} \\
	\le& 3\|\Omega\|_{L\log L(\sn)} + \frac1{\text{e}} \le 3(\|\Omega\|_{L\log L(\sn)}+1),
\end{align*}
where the second inequality follows from two basic facts that
\[\log^+|\Omega(\theta)| \le \log(e+|\Omega(\theta)|) \quad\text{and}\quad \alpha\log^+\frac1\alpha \le \frac1{\text{e}} \quad\text{for }\alpha\in(0,1].\]

As for $T_{\vec\Omega}$, by the H\"older's inequality for weak spaces \cite[p.16]{G2014}: for \(p_j\in(0,\infty)\), \(j=1,\cdots,k\), it holds that
\[\|f_1\cdots f_k\|_{L^{p,\infty}} \le p^{-\frac1p}\prod_{j=1}^kp_j^{\frac1{p_j}}\|f_j\|_{L^{p_j,\infty}},\quad\frac1p=\frac1{p_1}+\cdots+\frac1{p_k}.\]
This inequality, together with the pointwise control $|T_{\vec\Omega}(f_1,f_2)| \le |T_{\Omega_1}(f_1)||T_{\Omega_2}(f_2)|$ gives that
\begin{align*}
	\|T_{\vec\Omega}(f_1,f_2)\|_{L^{1/2,\infty}} &\le \big\||T_{\Omega_1}(f_1)||T_{\Omega_2}(f_2)|\big\|_{L^{1/2,\infty}} \\
	&\le \sqrt2\|T_{\Omega_1}(f_1)\|_{L^{1,\infty}}\|T_{\Omega_2}(f_2)\|_{L^{1,\infty}} \\
	&\le C\prod_{i=1}^2\mathcal C_{\Omega_i}\|f_i\|_{L^1(\rn)},
\end{align*}
which finishes the proof of Lemma \ref{lem23}.
\end{proof}

Now we state some basic properties of $L\log L$ space.

\begin{lem}\label{lem24}
If $\Phi_1(\theta),\Phi_2(\theta)\in L\log L(\sn)$, then they enjoy the following properties:
\begin{enumerate}
	\item[\rm{(i)}] $\Phi_1(\theta),\Phi_2(\theta)\in L^1(\sn)$, and for $i=1,2$, $\|\Phi_i\|_{L^1(\sn)}\le\|\Phi_i\|_{L\log L(\sn)};$
	\item[\rm{(ii)}] The quasi-triangle inequality is true in $L\log L$ space:
		$$\|\Phi_1+\Phi_2\|_{L\log L(\sn)}\le4(\|\Phi_1\|_{L\log L(\sn)}+\|\Phi_2\|_{L\log L(\sn)}).$$
\end{enumerate}
\end{lem}

\section{Proof of Theorem \ref{thm1}}\label{s3}

\subsection{Proof of Theorem \ref{thm1} (i)}\label{s31}

\begin{proof}


Without loss of generality, we may assume $\|f\|_{L^1(\rn)}=1$. The same assumption applies to the rest proofs of our Theorems.

For any $0<\varepsilon\ll\min\{1,\|\Omega\|_{L^1(\sn)}\}$, it is easy to see that there exists a real number $r_{\varepsilon}>1$, such that
$$\int_{B(0,r_{\varepsilon})}f(x)dx>1-\varepsilon.$$
Now let $g=f\chi_{B(0,r_{\varepsilon})},h=f\chi_{B(0,r_\varepsilon)^c}$. For $\lambda>0$, we denote
\begin{align*}
	E_\lambda &= \{x:|T_\Omega(f)(x)|>\lambda\}; \\
	E_\lambda^1 &= \{x:|T_\Omega(g)(x)|>\lambda\}; \\
	E_\lambda^2 &= \{x:|T_\Omega(h)(x)|>\lambda\}.
\end{align*}
Since $|T_\Omega|$ is sublinear, it follows that
$$E_{(1+\sqrt\varepsilon/2)\lambda}^1 \backslash E_{\sqrt\varepsilon\lambda/2}^2 \subset E_\lambda\subset E_{(1-\sqrt{\varepsilon}/2)\lambda}^1\cup E_{\sqrt{\varepsilon}\lambda/2}^2.$$


To set up the argument, we first give a decomposition for the rough kernel $\Omega$. Since $C(\sn)$ is dense in $L\log L(\sn)$, then there exists a continuous function $\widetilde\Omega_\varepsilon$ on $\sn$ such that
$$\|\Omega-\widetilde\Omega_\varepsilon\|_{L\log L(\sn)} < \frac\varepsilon{12}.$$
Take
$$\Omega_\varepsilon = \widetilde\Omega_\varepsilon - \frac1{\sigma(\sn)}\int_\sn\widetilde\Omega_\varepsilon(\theta)d\sigma(\theta).$$
Then both $\Omega_\varepsilon$ and $\Omega-\Omega_\varepsilon$ satisfy the vanishing condition again. By Lemma \ref{lem24}, we have
\begin{align*}
	\|\Omega-\Omega_\varepsilon\|_{L\log L(\sn)} =& \Big\| (\Omega-\widetilde\Omega_\varepsilon) + \frac1{\sigma(\sn)}\int_\sn\widetilde\Omega_\varepsilon(\theta)-\Omega(\theta)d\sigma(\theta) \Big\|_{L\log L(\sn)} \\
	\le& 4\left(\|\Omega-\widetilde\Omega_\varepsilon\|_{L\log L(\sn)} + \Big\|\frac{\|\Omega-\widetilde\Omega_\varepsilon\|_{L^1(\sn)}}{\sigma(\sn)}\Big\|_{L\log L(\sn)}\right) \\
	\le& 4\left(\frac\varepsilon{12} + \int_\sn\frac\varepsilon{12\sigma(\sn)}\log(e+\frac\varepsilon{12\sigma(\sn)}) \right) \le \varepsilon.
\end{align*}

On the other hand, since for all $0<\alpha<1$, $\alpha|\log\alpha| \le 4\alpha^{3/4}/\text{e}$ holds. Therefore
\begin{align*}
	\mathcal C_{\Omega-\Omega_\varepsilon} &\le 3\|\Omega-\Omega_\varepsilon\|_{L\log L(\sn)} + \|\Omega-\Omega_\varepsilon\|_{L^1(\sn)}\log^+\frac1{\|\Omega-\Omega_\varepsilon\|_{L^1(\sn)}} \\
	&\le 3\varepsilon + \frac4{\text{e}}\|\Omega-\Omega_\varepsilon\|_{L^1(\sn)}^{3/4} \le \left(3+\frac4{\text{e}}\right)\varepsilon^{3/4}.
\end{align*}

Now we denote
\[E_\lambda^{1,1} = \{x:|T_{\Omega_\varepsilon}(g)(x)|>\lambda\};\qquad E_\lambda^{1,2} = \{x:|T_{\Omega-\Omega_\varepsilon}(g)(x)|>\lambda\}.\]
One can easily deduce that
$$E_{(1+\sqrt\varepsilon)\lambda}^{1,1} \backslash E_{\sqrt\varepsilon\lambda/2}^{1,2} \subset E_{(1+\sqrt\varepsilon/2)\lambda}^1 \qquad\text{and}\qquad E_{(1-\sqrt\varepsilon/2)\lambda}^1 \subset E_{(1-\sqrt\varepsilon)\lambda}^{1,1} \cup E_{\sqrt\varepsilon\lambda/2}^{1,2},$$
which further implies that
\begin{equation}\label{eq3.1}
	E_{(1+\sqrt\varepsilon)\lambda}^{1,1} \backslash (E_{\sqrt\varepsilon\lambda/2}^{1,2} \cup E_{\sqrt\varepsilon\lambda/2}^2) \subset E_\lambda \subset E_{(1-\sqrt\varepsilon)\lambda}^{1,1} \cup (E_{\sqrt\varepsilon\lambda/2}^{1,2} \cup E_{\sqrt\varepsilon\lambda/2}^2).
\end{equation}

Here is the main idea of the proof of Theorem \ref{thm1} (i). To estimate $E_\lambda$, we need to give an upper estimate of right side in the above inequality and a lower estimate of left side in the above inequality. We split the proof into three parts. In Part 1, we will give the upper estimates of $|E_{\sqrt\varepsilon\lambda/2}^{1,2}|$ and $|E_{\sqrt\varepsilon\lambda/2}^2|$. Part 2 and Part 3 will be devoted to give the upper estimate and lower estimate of $|E_{(1-\sqrt\varepsilon)\lambda}^{1,1}|$ and $|E_{(1+\sqrt\varepsilon)\lambda}^{1,1}|$. In these two parts, the good things are that $g$ has compact support, $\Omega_\varepsilon$ is continuous function. Combining with the upper estimates in Part 1, we further give the upper estimate for $|E_\lambda|$ in Part 2. Moreover, the upper estimates in Part 1 and the lower estimate in Part 3 yield the lower estimate for $|E_\lambda|$ in Part 3.

\vspace{0.3cm}

\noindent\textbf{Part 1: Upper estimate for $|E_{\sqrt\varepsilon\lambda/2}^{1,2}|$ and $|E_{\sqrt\varepsilon\lambda/2}^2|$.}

By Lemma \ref{lem23}, one may get
$$\frac{\sqrt\varepsilon\lambda}2|E_{\sqrt\varepsilon\lambda/2}^{1,2}| \le C\mathcal C_{\Omega-\Omega_\varepsilon}\|g\|_{L^1} \le C\varepsilon^{3/4},$$
which indicates that
\begin{equation}\label{eq3.2}
	|E_{\sqrt\varepsilon\lambda/2}^{1,2}| \le C\frac{\varepsilon^{1/4}}\lambda.
\end{equation}

On the other hand, Lemma \ref{lem23} also yields that
$$\frac{\sqrt\varepsilon\lambda}2|E_{\sqrt\varepsilon\lambda/2}^2| \le C\mathcal C_\Omega\varepsilon \le C(\|\Omega\|_{L\log L(\sn)}+1)\varepsilon^{3/4},$$
which leads to
\begin{equation}\label{eq3.3}
	|E_{\sqrt\varepsilon\lambda/2}^2| \le C(\|\Omega\|_{L\log L(\sn)}+1)\frac{\varepsilon^{1/4}}\lambda.
\end{equation}

\vspace{0.3cm}

\noindent\textbf{Part 2: Upper estimate for $|E_{(1-\sqrt{\varepsilon})\lambda}^{1,1}|$.}

Since $\Omega_\varepsilon$ is continuous on $\sn$, then it is uniformly continuous on $\sn$. Hence there exists a real positive number $d_\varepsilon<\varepsilon$, such that, if $\sigma(\theta_1,\theta_2)<d_\varepsilon$, we have
$$|\Omega_\varepsilon(\theta_1)-\Omega_\varepsilon(\theta_2)| < \varepsilon,\quad\text{for }\theta_1,\theta_2\in\sn.$$
Now let $R_\varepsilon=r_\varepsilon/\arcsin(d_\varepsilon)$, then it's easy to see that for $|x|>R_\varepsilon,|y|\le r_\varepsilon$, it holds that
$$\sigma\Big(\frac{x}{|x|},\frac{x-y}{|x-y|}\Big) < d_\varepsilon.$$
See Figure \ref{omega} for 2-dimensional case.
\begin{figure}[h]
	\centering
	\includegraphics[width=12cm]{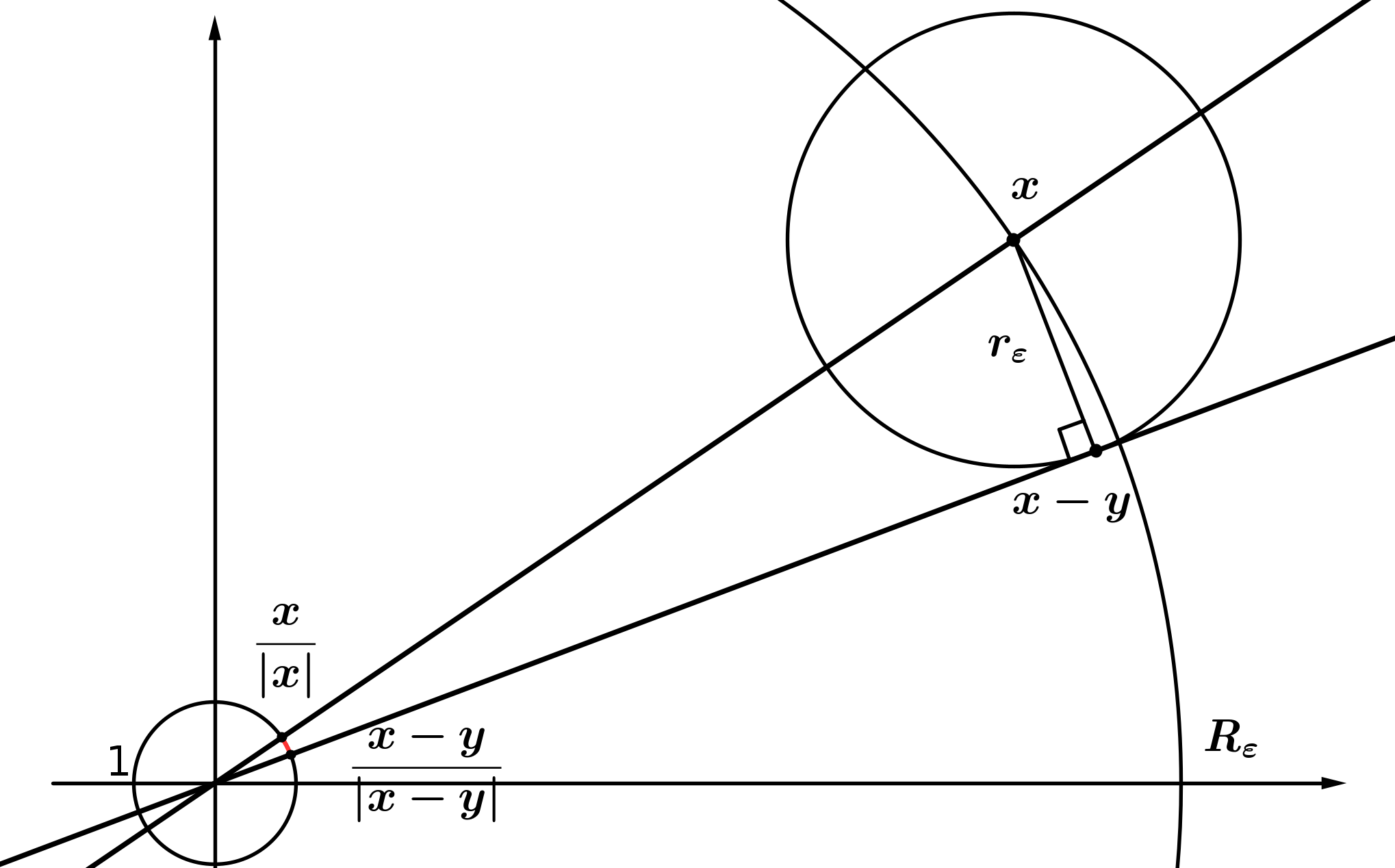}
	\caption{Suppose $\varepsilon$ is small enough, $|x|>R_\varepsilon$ and $|y|\le r_\varepsilon$. Despite the fact that $x$ and $x-y$ may be far away from each other, when we pull them back to the unit sphere, these two points are close enough.}
	\label{omega}
\end{figure}
Therefore
\begin{equation}\label{eq3.4}
	\Big|\Omega_\varepsilon\Big(\frac{x-y}{|x-y|}\Big)-\Omega_\varepsilon\Big(\frac{x}{|x|}\Big)\Big| < \varepsilon.
\end{equation}

When $|x|>R_\varepsilon$ and $y\in\text{supp }g$, it's easy to see that
$$|x-y| \ge |x|-r_\varepsilon > (1-d_\varepsilon/2)|x| > (1-\sqrt\varepsilon)|x|,$$
which means that
\begin{equation}\label{eq3.5}
	|T_{\vec\Omega_\varepsilon}(g_1,g_2)(x)| \le \int_{B(0,r_\varepsilon)}\frac{|\Omega_\varepsilon(\frac{x-y}{|x-y|})|}{|x-y|^n}g(y)dy \le \frac{|\Omega_\varepsilon(\frac x{|x|})|+\varepsilon}{(1-\sqrt\varepsilon)^n|x|^n}.
\end{equation}

Therefore, it follows from Lemma \ref{lem21} and (\ref{eq3.5}) that
\begin{align*}
	|E_{(1-\sqrt\varepsilon)\lambda}^{1,1}| &\le \Big|\Big\{|x|>R_\varepsilon:\frac{|\Omega_\varepsilon(\frac x{|x|})|+\varepsilon}{(1-\sqrt\varepsilon)^n|x|^n}>(1-\sqrt\varepsilon)\lambda\Big\}\Big| + |\overline{B(0,R_\varepsilon)}| \\
	&= \frac{\||\Omega_\varepsilon|+\varepsilon)\|_{L^1(\sn)}}{n(1-\sqrt\varepsilon)^{n+1}\lambda} + |\overline{B(0,R_\varepsilon)}|.
\end{align*}
Combining this estimate with (\ref{eq3.1}), (\ref{eq3.2}) and (\ref{eq3.3}), we obtain the upper estimate for $|E_\lambda|$:
\begin{align*}
	|E_\lambda| \le& |E_{(1-\sqrt\varepsilon)\lambda}^{1,1}| + |E_{\sqrt{\varepsilon}\lambda/2}^{1,2}| + |E_{\sqrt{\varepsilon}\lambda/2}^2| \\
	\le& \frac{\|\Omega_\varepsilon\|_{L^1(\sn)} + \sigma(\sn)\varepsilon}{n(1-\sqrt\varepsilon)^{n+1}\lambda} + |\overline{B(0,R_\varepsilon)}| + C(\|\Omega\|_{L\log L(\sn)} + 1)\frac{\varepsilon^{1/4}}\lambda.
\end{align*}
Now multiplying $\lambda$ on both sides of the above inequality and let $\lambda\to0^+$, we get
\begin{align*}
	\varlimsup_{\lambda\to0^+}\lambda|E_\lambda| &\le \frac{\|\Omega_\varepsilon\|_{L^1(\sn)}}{n(1-\sqrt\varepsilon)^{n+1}} + \frac{\sigma(\sn)\varepsilon}{n(1-\sqrt\varepsilon)^{n+1}} + C(\|\Omega\|_{L\log L(\sn)} + 1)\varepsilon^{1/4} \\
	&\le \frac{\|\Omega\|_{L^1(\sn)}+\varepsilon}{n(1-\sqrt\varepsilon)^{n+1}} + C(\|\Omega\|_{L\log L(\sn)} + 1)\varepsilon^{1/4}.
\end{align*}

Hence, by the arbitrariness of $\varepsilon$, it follows that
\begin{equation}\label{eq3.6}
	\varlimsup_{\lambda\to0^+}\lambda|E_\lambda| \le \frac{\|\Omega\|_{L^1(\sn)}}n,
\end{equation}
which gives the desired upper estimate.

\vspace{0.3cm}

\noindent{\bf Part 3: Lower estimate for $|E_{(1+\sqrt\varepsilon)\lambda}^{1,1}|$.}

When $|x|>R_\varepsilon$, $|y|\le r_\varepsilon$, by (\ref{eq3.4}), we have
$$\Big|\Omega_\varepsilon\Big(\frac x{|x|}\Big)\Big|-\varepsilon \le \Big|\Omega_\varepsilon\Big(\frac{x-y}{|x-y|}\Big)\Big|.$$
However, to give the lower control of $|T_{\Omega_\varepsilon}(g)(x)|$ by using the above inequality, we need to put the absolute value inside the integral. To overcome this obstacle, we introduce the following two auxiliary sets:
\begin{align*}
	S_\varepsilon :=& \{\theta\in\sn:|\Omega_\varepsilon(\theta)|>\varepsilon\}, \\
	V_\varepsilon :=& \big\{x\in\rn:\frac{x}{|x|}\in S_\varepsilon\big\}.
\end{align*}
Therefore for all $x\in V_\varepsilon\cap\overline{B(0,R_\varepsilon)}^c$ and $|y|\le r_\varepsilon$, $\Omega(\frac{x-y}{|x-y|})$ doesn't change its sign. Then we can give the lower control of $|T_{\Omega_\varepsilon}(g)(x)|$ as follows:
\begin{equation}\label{eq3.7}
	|T_{\vec\Omega_\varepsilon}(g_1,g_2)(x)| = \int_{B(0,r_\varepsilon)}\frac{|\Omega_\varepsilon(\frac{x-y}{|x-y|})|}{|x-y|^n}g(y)dy \ge \frac{(1-\varepsilon)(|\Omega_\varepsilon(\frac x{|x|})|-\varepsilon)}{(1+\sqrt\varepsilon)^n|x|^n},
\end{equation}
where the last inequality follows from
$$|x-y| \le |x|+r_\varepsilon < (1+d_\varepsilon/2)|x| < (1+\sqrt\varepsilon)|x|.$$

Lemma \ref{lem21} together with (\ref{eq3.7}) may lead to
\begin{align*}
	|E_{(1+\sqrt\varepsilon)\lambda}^{1,1}| \ge& |E_{(1+\sqrt\varepsilon)\lambda}^{1,1} \cap V_\varepsilon| \\
	\ge& \Big|\Big\{x\in V_\varepsilon:\frac{(1-\varepsilon)(|\Omega_\varepsilon(\frac x{|x|})|-\varepsilon)}{(1+\sqrt\varepsilon)^n|x|^n} > (1+\sqrt\varepsilon)\lambda\Big\}\Big| - |\overline{B(0,R_\varepsilon)}| \\
	=& \frac{(1-\varepsilon)\||\Omega_\varepsilon|-\varepsilon\|_{L^1(S_\varepsilon)}}{n(1+\sqrt\varepsilon)^{n+1}\lambda} - |\overline{B(0,R_\varepsilon)}|,
\end{align*}
which, combining with (\ref{eq3.1}), (\ref{eq3.2}) and (\ref{eq3.3}) further gives that
\begin{align*}
	|E_\lambda| \ge& |E_{(1+\sqrt\varepsilon)\lambda}^{1,1}| - |E_{\sqrt\varepsilon\lambda/2}^{1,2}| - |E_{\sqrt\varepsilon\lambda/2}^2| \\
	\ge& \frac{(1-\varepsilon)(\||\Omega_\varepsilon|\|_{L^1(S_\varepsilon)} - \sigma(S_\varepsilon)\varepsilon)}{n(1+\sqrt\varepsilon)^{n+1}\lambda} - |\overline{B(0,R_\varepsilon)}| - C(\|\Omega\|_{L\log L(\sn)} + 1)\frac{\varepsilon^{1/4}}\lambda,
\end{align*}
where we used an obvious fact that $\||\Omega_\varepsilon|-\varepsilon\|_{L^1(S_\varepsilon)} = \||\Omega_\varepsilon|\|_{L^1(S_\varepsilon)} - \sigma(S_\varepsilon)\varepsilon$ since $|\Omega_\varepsilon|$ is bigger that $\varepsilon$ on $S_\varepsilon$. Multiplying $\lambda$ on both sides and let $\lambda\to0^+$, we obtain that
\begin{align*}
	\varliminf_{\lambda\to0^+}\lambda|E_\lambda| &\ge \frac{(1-\varepsilon)(\|\Omega_\varepsilon\|_{L^1(S_\varepsilon)} - \sigma(S_\varepsilon)\varepsilon)}{n(1+\sqrt\varepsilon)^{n+1}} - C(\|\Omega\|_{L\log L(\sn)} + 1)\varepsilon^{1/4} \\
	&= \frac{(1-\varepsilon)\|\Omega_\varepsilon\|_{L^1(\sn)}}{n(1+\sqrt\varepsilon)^{n+1}} - \frac{\|\Omega_\varepsilon\|_{L^1(\sn\backslash S_\varepsilon)} + \sigma(S_\varepsilon)\varepsilon}{n(1+\sqrt\varepsilon)^{n+1}}\\
	&\quad - C(\|\Omega\|_{L\log L(\sn)} + 1)\varepsilon^{1/4} \\
	&\ge \frac{(1-\varepsilon)(\|\Omega\|_{L^1(\sn)}-\varepsilon)}{n(1+\sqrt\varepsilon)^{n+1}} - \frac{\sigma(\sn)\varepsilon}{n(1+\sqrt\varepsilon)^{n+1}} - C(\|\Omega\|_{L\log L(\sn)} + 1)\varepsilon^{1/4} .
\end{align*}

Since $\varepsilon$ is arbitrary and small enough, we get
\begin{equation}\label{eq3.8}
	\varliminf_{\lambda\to0^+}\lambda|E_\lambda| \ge \frac{\|\Omega\|_{L^1(\sn)}}n.
\end{equation}

Finally, from (\ref{eq3.6}) and (\ref{eq3.8}), we obtain that
$$\lim_{\lambda\to0^+}\lambda|E_\lambda|=\frac{\|\Omega\|_{L^1(\sn)}}n.$$
\end{proof}

\subsection{Proof of Theorem \ref{thm1} (ii)}\label{s32}

\begin{proof} For $\lambda>0$, we set
\begin{align*}
	F_\lambda =& \bigg\{x:\Big||T_\Omega(f)(x)| - \frac{|\Omega(\frac x{|x|})|}{|x|^n}\Big|>\lambda\bigg\}; \\
	F^1_\lambda =& \bigg\{|x|>R_\varepsilon:\Big||T_{\Omega_\varepsilon}(g)(x)|-\frac{|\Omega_\varepsilon(\frac x{|x|})|}{|x|^n}\Big|>\lambda\bigg\}; \\
	F_\lambda^2 =& \bigg\{x:\Big|\frac{|\Omega(\frac x{|x|})|-|\Omega_\varepsilon(\frac x{|x|})|}{|x|^n}\Big|>\lambda\bigg\}.
\end{align*}
Now we claim that
\begin{equation}\label{eq3.9}
	F_\lambda \subset F^1_{(1-2\sqrt\varepsilon)\lambda} \cup F_{\sqrt\varepsilon\lambda}^2 \cup E_{\sqrt{\varepsilon}\lambda/2}^{1,2} \cup E_{\sqrt{\varepsilon}\lambda/2}^2 \cup \overline{B(0,R_\varepsilon)}.
\end{equation}

To see this, it is sufficient to show the complementary set of right side is contained in $F_\lambda^c$. For any $x$ in the complementary set of right side, we have $|x|>R_\varepsilon$ and the following facts:
$$\Big|\frac{|\Omega(\frac x{|x|})|-|\Omega_\varepsilon(\frac x{|x|})|}{|x|^n}\Big| \le \sqrt\varepsilon\lambda;$$
$$|T_\Omega(h)(x)|, |T_{\Omega-\Omega_\varepsilon}(g)(x)| \le \frac{\sqrt\varepsilon}2\lambda;$$
$$\Big||T_{\Omega_\varepsilon}(g)(x)| - \frac{|\Omega_\varepsilon(\frac x{|x|})|}{|x|^n}\Big| \le (1-2\sqrt\varepsilon)\lambda.$$

From these inequalities, it's easy to deduce that
\begin{align*}
	|T_\Omega(f)(x)| &\le |T_\Omega(g)(x)| + \frac{\sqrt\varepsilon}2\lambda \le |T_{\Omega_\varepsilon}(g)(x)| + \sqrt\varepsilon\lambda \\
	&\le \frac{|\Omega_\varepsilon(\frac x{|x|})|}{|x|^n} + (1-\sqrt\varepsilon)\lambda \le \frac{|\Omega(\frac x{|x|})|}{|x|^n} + \lambda,
\end{align*}
which shows that $x\in F_\lambda^c$. This finishes the proof of the claim (\ref{eq3.9}).

Since the estimates of $|E_{\sqrt\varepsilon\lambda/2}^{1,2}|$ and $|E_{\sqrt\varepsilon\lambda/2}^2|$ have already been established in Subsection \ref{s31}, it is sufficient to consider the estimates of $|F^1_{(1-2\sqrt\varepsilon)\lambda}|$ and $|F_{\sqrt\varepsilon\lambda}^2|$. We split the proof into three parts.

\vspace{0.3cm}

\noindent\textbf{Part 1: Upper estimate for $|F_{\sqrt\varepsilon\lambda}^2|$.}

Applying Lemma \ref{lem21}, we conclude that
\begin{equation}\label{eq3.10}
	|F_{\sqrt\varepsilon\lambda}^2| = \frac{\||\Omega| - |\Omega_\varepsilon|\|_{L^1(\sn)}}{n\sqrt\varepsilon\lambda} \le \frac{\sqrt\varepsilon}{n\lambda} \le C\frac{\varepsilon^{1/4}}\lambda.
\end{equation}

\vspace{0.3cm}

\noindent\textbf{Part 2: Upper estimate for $|F^1_{(1-2\sqrt\varepsilon)\lambda} \cap V_\varepsilon|$.}

For $x\in F^1_{(1-2\sqrt\varepsilon)\lambda} \cap V_\varepsilon$, from (\ref{eq3.5}) and (\ref{eq3.7}), we know that $|T_{\Omega_\varepsilon}(g)(x)|$ and $\frac{|\Omega_\varepsilon(\frac x{|x|})|}{|x|^n}$ are between $\frac{(1-\varepsilon)(|\Omega_\varepsilon(\frac x{|x|})|-\varepsilon)}{(1+\sqrt\varepsilon)^n|x|^n}$ and $\frac{|\Omega_\varepsilon(\frac x{|x|})|+\varepsilon}{(1-\varepsilon)^n|x|^n}$, which implies that
$$\Big||T_{\Omega_\varepsilon}(g)(x)| - \frac{|\Omega_\varepsilon(\frac x{|x|})|}{|x|^n}\Big| \le \frac{|\Omega_\varepsilon(\frac x{|x|})|+\varepsilon}{(1-\sqrt\varepsilon)^n|x|^n} - \frac{(1-\varepsilon)(|\Omega_\varepsilon(\frac x{|x|})|-\varepsilon)}{(1+\sqrt\varepsilon)^n|x|^n} =: \frac{I_{\vec\Omega,\varepsilon}(\frac x{|x|})}{|x|^{2n}}.$$
Thus
$$|F_{(1-2\sqrt\varepsilon)\lambda}^1 \cap V_\varepsilon| \le \Big|\Big\{x\in V_\varepsilon:\frac{I_{\vec\Omega,\varepsilon}(\frac x{|x|})}{|x|^n}>(1-2\sqrt\varepsilon)\lambda\Big\}\Big| = \frac{\|I_{\vec\Omega,\varepsilon}\|_{L^1(S_\varepsilon)}}{n(1-2\sqrt\varepsilon)\lambda}.$$

It is easy to see that
\begin{align*}
	& \|I_{\vec\Omega,\varepsilon}\|_{L^1(S_\varepsilon)} \\
	&= \|\Omega_\varepsilon\|_{L^1(S_\varepsilon)}\left(\frac1{(1-\sqrt\varepsilon)^n}-\frac{1-\varepsilon}{(1+\sqrt\varepsilon)^n}\right) + \left(\frac1{(1-\sqrt\varepsilon)^n}+\frac{1-\varepsilon}{(1+\sqrt\varepsilon)^n}\right)\sigma(S_\varepsilon)\varepsilon \\
	&\le C(\|\Omega\|_{L^1(\sn)}+1)\sqrt\varepsilon \le C(\|\Omega\|_{L\log L(\sn)}+1)\varepsilon^{1/4}.
\end{align*}
Therefore
\begin{equation}\label{eq3.11}
	|F_{(1-2\sqrt\varepsilon)\lambda}^1 \cap V_\varepsilon| \le C(\|\Omega\|_{L\log L(\sn)}+1)\frac{\varepsilon^{1/4}}\lambda.
\end{equation}

\vspace{0.3cm}

\noindent{\bf Part 3: Upper estimate for $|F_{(1-2\sqrt\varepsilon)\lambda}^1 \cap V_\varepsilon^c|$.}

For all $|x|>R_\varepsilon$, it follows from the inequality (\ref{eq3.5}) that
$$\Big||T_{\Omega_\varepsilon}(g)(x)| - \frac{|\Omega_\varepsilon(\frac x{|x|})|}{|x|^n}\Big| \le 2\frac{|\Omega_\varepsilon(\frac x{|x|})|+\varepsilon}{(1-\sqrt\varepsilon)^n|x|^n} \le 2\frac{|\Omega_\varepsilon(\frac x{|x|})|+\varepsilon}{(1-2\sqrt\varepsilon)^n|x|^n}.$$
Then by Lemma \ref{lem21}, $|F_{(1-2\sqrt\varepsilon)\lambda}^1\cap V_\varepsilon^c|$ is dominated by
\begin{equation}\label{eq3.12}
	|F_{(1-2\sqrt\varepsilon)\lambda}^1\cap V_\varepsilon^c| \le \frac{2\||\Omega_\varepsilon|+\varepsilon\|_{L^1(\sn\backslash S_\varepsilon)}}{n(1-2\sqrt\varepsilon)^{n+1}\lambda} \le \frac{4\sigma(\sn\backslash S_\varepsilon)\varepsilon}{n(1-2\sqrt\varepsilon)^{n+1}\lambda} \le C\frac{\varepsilon^{1/4}}\lambda.
\end{equation}

Finally, by (\ref{eq3.2}), (\ref{eq3.3}), (\ref{eq3.9}), (\ref{eq3.10}), (\ref{eq3.11}) and (\ref{eq3.12}), it holds that
\begin{align*}
	|F_\lambda| \le C(\|\Omega\|_{L\log L(\sn)} + 1)\frac{\varepsilon^{1/4}}\lambda.
\end{align*}
Multiplying $\lambda$ on both sides, and by the arbitrariness of $\varepsilon$, we finally get
$$\lim_{\lambda\to0^+}\lambda|F_\lambda|=0.$$
The proof of Theorem \ref{thm1} is finished.

\end{proof}

\section{Proof of Theorem \ref{thm2}}\label{s4}

\subsection{Proof of Theorem \ref{thm2} (i)}\label{s41}

\begin{proof}
We still assume $\|f_i\|_{L^1}=1$ for $i=1,2$, $0<\varepsilon\ll\min\{1,\|\Omega_1\|_{L^1(\sn)},\|\Omega_2\|_{L^1(\sn)}\}$, and there exists a real number $\widetilde r_{\varepsilon}>1$ such that
$$\int_{B(0,\widetilde r_{\varepsilon})}f_i(x)dx>1-\varepsilon,\quad\text{for }i=1,2.$$
Set $g_i=f_i\chi_{B(0,\widetilde r_{\varepsilon})},h_i=f_i\chi_{B(0,\widetilde r_\varepsilon)^c}$. For $\lambda>0$, we denote
\begin{align*}
	G_\lambda &= \{x:|T_{\vec\Omega}(f_1,f_2)(x)|>\lambda\}; \\
	G_\lambda^1 &= \{x:|T_{\vec\Omega}(g_1,g_2)(x)|>\lambda\}; \\
	G_\lambda^2 &= \{x:|T_{\vec\Omega}(g_1,h_2)(x)|+|T_{\vec\Omega}(h_1,g_2)(x)|+|T_{\vec\Omega}(h_1,h_2)(x)|>\lambda\}; \\
	G_\lambda^3 &= \{x:|T_{\vec\Omega}(f_1,h_2)(x)|+|T_{\vec\Omega}(h_1,f_2)(x)|+|T_{\vec\Omega}(h_1,h_2)(x)|>\lambda\}.
\end{align*}
Since $|T_{\vec\Omega}|$ is sublinear, it follows that
$$G_{(1+\sqrt\varepsilon/2)\lambda}^1 \backslash G_{\sqrt\varepsilon\lambda/2}^3 \subset G_\lambda\subset G_{(1-\sqrt{\varepsilon}/2)\lambda}^1\cup G_{\sqrt{\varepsilon}\lambda/2}^2.$$

Now we decompose the rough kernel $\vec\Omega$. Using the same method as in Subsection \ref{s31}, there exists two continuous functions $\Omega_{1,\varepsilon}$ and $\Omega_{2,\varepsilon}$ on $\sn$ satisfying the vanishing condition and
$$\|\Omega_i-\Omega_{i,\varepsilon}\|_{L\log L(\sn)} < \varepsilon\quad\text{and}\quad\mathcal C_{\Omega_i-\Omega_{i,\varepsilon}}\le \left(3+\frac4{\text{e}}\right)\varepsilon^{3/4}\quad\text{for }i=1,2.$$
Denote
\begin{align*}
	& \vec\Omega_\varepsilon = (\Omega_{1,\varepsilon},\Omega_{2,\varepsilon}); \\
	& G_\lambda^{1,1} = \{ x:|T_{\vec\Omega_\varepsilon}(g_1,g_2)(x)|>\lambda\}; \\
	& G_\lambda^{1,2} = \{ x:|T_{(\Omega_{1,\varepsilon},\Omega_2-\Omega_{2,\varepsilon})}(g_1,g_2)(x)|+|T_{(\Omega_1-\Omega_{1,\varepsilon},\Omega_{2,\varepsilon})}(g_1,g_2)(x)| \\
	&\qquad\qquad +|T_{(\Omega_1-\Omega_{1,\varepsilon},\Omega_2-\Omega_{2,\varepsilon})}(g_1,g_2)(x)|>\lambda\}; \\
	& G_\lambda^{1,3} = \{ x:|T_{(\Omega_1,\Omega_2-\Omega_{2,\varepsilon})}(g_1,g_2)(x)|+|T_{(\Omega_1-\Omega_{1,\varepsilon},\Omega_2)}(g_1,g_2)(x)| \\
	&\qquad\qquad +|T_{(\Omega_1-\Omega_{1,\varepsilon},\Omega_2-\Omega_{2,\varepsilon})}(g_1,g_2)(x)|>\lambda\}.
\end{align*}
One can easily deduce that
$$G_{(1+\sqrt\varepsilon)\lambda}^{1,1} \backslash G_{\sqrt\varepsilon\lambda/2}^{1,3} \subset G_{(1+\sqrt\varepsilon/2)\lambda}^1 \qquad\text{and}\qquad G_{(1-\sqrt\varepsilon/2)\lambda}^1 \subset G_{(1-\sqrt\varepsilon)\lambda}^{1,1} \cup E_{\sqrt\varepsilon\lambda/2}^{1,2}.$$
Then it follows that
\begin{equation}\label{eq4.1}
	G_{(1+\sqrt\varepsilon)\lambda}^{1,1} \backslash (G_{\sqrt\varepsilon\lambda/2}^3 \cup G_{\sqrt\varepsilon\lambda/2}^{1,3}) \subset G_\lambda \subset G_{(1-2\sqrt\varepsilon/3)\lambda}^{1,1} \cup (G_{\sqrt\varepsilon\lambda/2}^2 \cup E_{\sqrt\varepsilon\lambda/3}^{1,2}).
\end{equation}

Similarly, we split the proof into three parts. In Part 1, we will give the upper estimates of $|G_{\sqrt\varepsilon\lambda/2}^{1,2}|$, $|G_{\sqrt\varepsilon\lambda/2}^{1,3}|$, $|G_{\sqrt\varepsilon\lambda/2}^2|$ and $|G_{\sqrt\varepsilon\lambda/2}^3|$. Part 2 and Part 3 will be devoted to give the upper estimate and lower estimate of $|G_{(1-\sqrt\varepsilon)\lambda}^{1,1}|$ and $|G_{(1+\sqrt\varepsilon)\lambda}^{1,1}|$.

\vspace{0.3cm}

\noindent\textbf{Part 1: Upper estimate for $|G_{\sqrt\varepsilon\lambda/2}^2|$, $|G_{\sqrt\varepsilon\lambda/2}^3|$, $|G_{\sqrt\varepsilon\lambda/2}^5|$ and $|G_{\sqrt\varepsilon\lambda/2}^6|$.}

By Lemma \ref{lem23}, one may get
\begin{align*}
	& \frac{\sqrt\varepsilon\lambda}6|\{x:|T_{\vec\Omega}(g_1,h_2)(x)|>\frac{\sqrt\varepsilon\lambda}6\}|^2 \le C\prod_{i=1}^2\mathcal C_{\Omega_i}\|g_1\|_{L^1}\|h_2\|_{L^1} \\
	&\le C\prod_{i=1}^2(\|\Omega_i\|_{L\log L(\sn)}+1)\varepsilon \le C\prod_{i=1}^2(\|\Omega_i\|_{L\log L(\sn)}+1)\varepsilon^{3/4},
\end{align*}
which implies
$$|\{x:|T_{\vec\Omega}(g_1,h_2)(x)|>\frac{\sqrt\varepsilon\lambda}6\}| \le C\prod_{i=1}^2(\|\Omega_i\|_{L\log L(\sn)}+1)^{1/2}\frac{\varepsilon^{1/8}}{\lambda^{1/2}}.$$
Similar estimate holds for $|T_{\vec\Omega}(h_1,g_2)|$, $|T_{\vec\Omega}(h_1,h_2)|$, $|T_{\vec\Omega}(f_1,h_2)|$ and $|T_{\vec\Omega}(h_1,f_2)|$. These estimates together with the definitions of $G_{\sqrt\varepsilon\lambda/2}^2$ and $G_{\sqrt\varepsilon\lambda/2}^3$ yield that
\begin{equation}\label{eq4.2}
	|G_{\sqrt\varepsilon\lambda/2}^2|, |G_{\sqrt\varepsilon\lambda/2}^3| \le C\prod_{i=1}^2(\|\Omega_i\|_{L\log L(\sn)}+1)^{1/2}\frac{\varepsilon^{1/8}}{\lambda^{1/2}}.
\end{equation}

On the other hand, Lemma \ref{lem23} also yields that
\begin{align*}
	& \frac{\sqrt\varepsilon\lambda}6|\{ x:|T_{(\Omega_{1,\varepsilon},\Omega_2-\Omega_{2,\varepsilon})}(g_1,g_2)(x)|>\frac{\sqrt\varepsilon\lambda}6\}|^2 \\
	&\le C\prod_{i=1}^2\|g_i\|_{L^1}\mathcal C_{\Omega_{1,\varepsilon}}\mathcal C_{\Omega_2-\Omega_{2,\varepsilon}} \\
	&\le C(\|\Omega_{1,\varepsilon}\|_{L\log L(\sn)}+1)\varepsilon^{3/4} \\
	&\le C(4\|\Omega_1\|_{L\log L(\sn)}+4\|\Omega_1-\Omega_{1,\varepsilon}\|_{L\log L(\sn)}+1)\varepsilon^{3/4} \\
	&\le C(\|\Omega_1\|_{L\log L(\sn)}+1)\varepsilon^{3/4}.
\end{align*}
Sigular result still holds for $|T_{(\Omega_1-\Omega_{1,\varepsilon},\Omega_{2,\varepsilon})}|$, $|T_{(\Omega_1-\Omega_{1,\varepsilon},\Omega_2-\Omega_{2,\varepsilon})}|$, $|T_{(\Omega_1,\Omega_2-\Omega_{2,\varepsilon})}|$ and $|T_{(\Omega_1-\Omega_{1,\varepsilon},\Omega_2)}|$. These estimates and the definitions of $G_{\sqrt\varepsilon\lambda/2}^{1,2}$, $G_{\sqrt\varepsilon\lambda/2}^{1,3}$ lead to
\begin{equation}\label{eq4.3}
	|G_{\sqrt\varepsilon\lambda/2}^{1,2}|, |G_{\sqrt\varepsilon\lambda/2}^{1,3}| \le C\sum_{i=1}^2\left(\|\Omega_i\|_{L\log L(\sn)}+1\right)^{1/2}\frac{\varepsilon^{1/8}}{\lambda^{1/2}}.
\end{equation}

\vspace{0.3cm}

\noindent\textbf{Part 2: Upper estimate for $|G_{(1-\sqrt{\varepsilon})\lambda}^{1,1}|$.}

By the same arugment as in Part 2 of Subsection \ref{s31}, there exists a real positive number $\widetilde d_\varepsilon<\varepsilon$, such that, if $\sigma(\theta_1,\theta_2)<\widetilde d_\varepsilon$, we have
$$|\Omega_{i,\varepsilon}(\theta_1)-\Omega_{i,\varepsilon}(\theta_2)| < \varepsilon,\quad\text{for }\theta_1,\theta_2\in\sn\text{ and }i=1,2.$$
Denote $\widetilde R_\varepsilon=\widetilde r_\varepsilon/\arcsin(\widetilde d_\varepsilon)$, then it's easy to see that for $|x|>\widetilde R_\varepsilon,|y_i|\le \widetilde r_\varepsilon$, it holds that
\begin{equation}\label{eq4.4}
	\left|\Omega_{i,\varepsilon}\left(\frac{x-y_i}{|x-y_i|}\right)-\Omega_{i,\varepsilon}\left(\frac{x}{|x|}\right)\right| < \varepsilon.
\end{equation}

Therefore for $|x|>\widetilde R_\varepsilon$, the same reasoning as in Part 2 of Subsection \ref{s31} gives that
\begin{equation}\label{eq4.5}
	|T_{\vec\Omega_\varepsilon}(g_1,g_2)(x)| \le \frac{(|\Omega_{1,\varepsilon}(\frac x{|x|})|+\varepsilon)(|\Omega_{2,\varepsilon}(\frac x{|x|})|+\varepsilon)}{(1-\sqrt\varepsilon)^{2n}|x|^{2n}}.
\end{equation}

Hence, by Lemma \ref{lem21} and (\ref{eq4.5}), we have
\begin{align*}
	|G_{(1-\sqrt\varepsilon)\lambda}^{1,1}| &\le \Bigg|\Bigg\{|x|>\widetilde R_\varepsilon:\frac{\prod\limits_{i=1}^2(|\Omega_{i,\varepsilon}(\frac x{|x|})|+\varepsilon)}{(1-\sqrt\varepsilon)^{2n}|x|^{2n}}>(1-\sqrt\varepsilon)\lambda\Bigg\}\Bigg| + |\overline{B(0,\widetilde R_\varepsilon)}| \\
	&= \frac{\Big\|\prod\limits_{i=1}^2(|\Omega_{i,\varepsilon}|+\varepsilon)\Big\|_{L^{1/2}(\sn)}^{1/2}}{n(1-\sqrt\varepsilon)^{n+1/2}\lambda^{1/2}} + |\overline{B(0,\widetilde R_\varepsilon)}|.
\end{align*}
Combining this estimate with (\ref{eq4.1}), (\ref{eq4.2}) and (\ref{eq4.3}), multiplying $\lambda^{1/2}$ on both sides of the above inequality and let $\lambda\to0^+$, we obtain the upper estimate as follows:
\begin{align*}
	\varlimsup_{\lambda\to0^+}\lambda^{1/2}|G_\lambda| \le& \frac{\Big\|\prod\limits_{i=1}^2(|\Omega_{i,\varepsilon}|+\varepsilon)\Big\|_{L^{1/2}(\sn)}^{1/2}}{n(1-\sqrt\varepsilon)^{n+1/2}}  \\
	&+ C\Big(\prod_{i=1}^2(\|\Omega_i\|_{L\log L(\sn)}+1)^{1/2} + \sum_{i=1}^2(\|\Omega_i\|_{L\log L(\sn)}+1)^{1/2}\Big)\varepsilon^{1/8}.
\end{align*}

Thus it remains to prove
\begin{equation}\label{eq4.6}
	\lim_{\varepsilon\to0^+}\Big\|\prod\limits_{i=1}^2(|\Omega_{i,\varepsilon}|+\varepsilon)\Big\|_{L^{1/2}(\sn)}^{1/2} = \|\Omega_1\Omega_2\|_{L^{1/2}(\sn)}^{1/2}.
\end{equation}
But there is no triangle inequality on the space $L^{1/2}(\sn)$. As a matter of fact, it enjoys the quasi-triangle inequality, which will produce a constant $2$. This is a really problem because this constant can not be ignored. Fortunately, if we add a power of $1/2$ on the outside of the $L^{1/2}(\sn)$ norm, this problem degenerates to an absolute value's triangle inequality with power $1/2$. This is exactly what we need. Hence, we have
\begin{align*}
	& \big|\|(|\Omega_{1,\varepsilon}|+\varepsilon)(|\Omega_{2,\varepsilon}|+\varepsilon)\|_{L^{1/2}(\sn)}^{1/2} - \|\Omega_1\Omega_2\|_{L^{1/2}(\sn)}^{1/2}\big| \\
	&\le \big\|(|\Omega_{1,\varepsilon}|+\varepsilon)(|\Omega_{2,\varepsilon}|+\varepsilon)-|\Omega_1\Omega_2|\big\|_{L^{1/2}(\sn)}^{1/2}.
\end{align*}
Applying this technique again, together with the H\"older's inequality, one may obtain
\begin{align*}
	& \big\|(|\Omega_{1,\varepsilon}|+\varepsilon)(|\Omega_{2,\varepsilon}|+\varepsilon) - |\Omega_1\Omega_2|\big\|_{L^{1/2}(\sn)}^{1/2} \\
	&\le \big\|(|\Omega_1-\Omega_{1,\varepsilon}|+\varepsilon)(|\Omega_{2,\varepsilon}|+\varepsilon)\big\|_{L^{1/2}(\sn)}^{1/2} + \big\||\Omega_1|(|\Omega_2-\Omega_{2,\varepsilon}|+\varepsilon)\big\|_{L^{1/2}(\sn)}^{1/2} \\
	&\le (\|\Omega_1-\Omega_{1,\varepsilon}\|_{L^1}+\sigma(\sn)\varepsilon)^{1/2}(\|\Omega_{2,\varepsilon}\|_{L^1}+\sigma(\sn)\varepsilon)^{1/2} \\
	&\quad +\|\Omega_1\|_{L^1}^{1/2}(\|\Omega_2-\Omega_{2,\varepsilon}\|_{L^1}+\sigma(\sn)\varepsilon)^{1/2}.
\end{align*}
By Lemma \ref{lem24} (i), it yields that
\begin{align*}
	& \big|\|(|\Omega_{1,\varepsilon}|+\varepsilon)(|\Omega_{2,\varepsilon}|+\varepsilon)\|_{L^{1/2}(\sn)}^{1/2} - \|\Omega_1\Omega_2\|_{L^{1/2}(\sn)}^{1/2}\big| \\
	&\le (1+\sigma(\sn))^{1/2}\varepsilon^{1/2}\cdot\sqrt2\sum_{i=1}^2\|\Omega_i\|_{L\log L(\sn)}^{1/2}.
\end{align*}
The right side of this inequality converges to $0$ as $\varepsilon\to0^+$, which implies that inequality (\ref{eq4.6}) holds.

Hence, by the arbitrariness of $\varepsilon$, it follows that
\begin{equation}\label{eq4.7}
	\varlimsup_{\lambda\to0^+}\lambda^{1/2}|G_\lambda| \le \frac{\|\Omega_1\Omega_2\|_{L^{1/2}(\sn)}^{1/2}}{n},
\end{equation}
which gives the desired upper estimate.

\vspace{0.3cm}

\noindent{\bf Part 3: Lower estimate for $|G_{(1+\sqrt\varepsilon)\lambda}^{1,1}|$.}

Here we also introduce two auxiliary sets as follows:
\begin{align*}
	\widetilde S_\varepsilon :=& \{\theta\in\sn:|\Omega_{1,\varepsilon}(\theta)|,|\Omega_{2,\varepsilon}(\theta)|>\varepsilon\}, \\
	\widetilde V_\varepsilon :=& \big\{x\in\rn:\frac{x}{|x|}\in S_\varepsilon\big\}.
\end{align*}
Then for $x\in V_\varepsilon\cap\overline{B(0,R_\varepsilon)}^c$, it follows that
\begin{equation}\label{eq4.8}
|T_{\vec\Omega_\varepsilon}(g_1,g_2)(x)| \ge \frac{(1-\varepsilon)(|\Omega_{1,\varepsilon}(\frac x{|x|})|-\varepsilon)(|\Omega_{2,\varepsilon}(\frac x{|x|})|-\varepsilon)}{(1+\sqrt\varepsilon)^{2n}|x|^{2n}}.
\end{equation}

Similarly, Lemma \ref{lem21} together with (\ref{eq4.8}) may lead to
\begin{align*}
	|G_{(1+\sqrt\varepsilon)\lambda}^{1,1}| \ge& |G_{(1+\sqrt\varepsilon)\lambda}^{1,1} \cap \widetilde V_\varepsilon| \\
	\ge& \Bigg|\Bigg\{ x\in \widetilde V_\varepsilon:\frac{(1-\varepsilon)\prod\limits_{i=1}^2(|\Omega_{i,\varepsilon}(\frac x{|x|})|-\varepsilon)}{(1+\sqrt\varepsilon)^{2n}|x|^{2n}} > (1+\sqrt\varepsilon)\lambda\Bigg\}\Bigg| - |\overline{B(0,R_\varepsilon)}| \\
	=& \frac{(1-\varepsilon)^{1/2}\Big\|\prod\limits_{i=1}^2(|\Omega_{i,\varepsilon}|-\varepsilon)\Big\|_{L^{1/2}(\widetilde S_\varepsilon)}^{1/2}}{n(1+\sqrt\varepsilon)^{n+1/2}\lambda^{1/2}} - |\overline{B(0,R_\varepsilon)}|,
\end{align*}
which, combining with (\ref{eq4.1}), (\ref{eq4.2}) and (\ref{eq4.3}), multiplying $\lambda^{1/2}$ on both sides and let $\lambda\to0^+$ further gives that
\begin{align*}
	\varliminf_{\lambda\to0^+}\lambda^{1/2}|G_\lambda| \ge& \frac{(1-\varepsilon)^{1/2}\Big\|\prod\limits_{i=1}^2(|\Omega_{i,\varepsilon}|-\varepsilon)\Big\|_{L^{1/2}(\widetilde S_\varepsilon)}^{1/2}}{n(1+\sqrt\varepsilon)^{n1/2}} \\
	& - C\Big(\prod_{i=1}^2(\|\Omega_i\|_{L\log L(\sn)}+1)^{1/2} + \sum_{i=1}^2(\|\Omega_i\|_{L\log L(\sn)}+1)^{1/2}\Big)\varepsilon^{1/8}.
\end{align*}

Now we are in a position to show that
\begin{equation}\label{eq4.9}
	\lim_{\varepsilon\to0}\Big\|\prod_{i=1}^2(|\Omega_{i,\varepsilon}|-\varepsilon)\Big\|_{L^{1/2}{(\widetilde S_\varepsilon)}}^{1/2} = \|\Omega_1\Omega_2\|_{L^{1/2}(\sn)}^{1/2}.
\end{equation}
Indeed, using the similar arguments as in Part 2, we obtain
\begin{align*}
	& \big|\|(|\Omega_{1,\varepsilon}|-\varepsilon)(|\Omega_{2,\varepsilon}|-\varepsilon)\|_{L^{1/2}(\sn)}^{1/2}-\|\Omega_1\Omega_2\|_{L^{1/2}(\sn)}^{1/2}\big| \\
	&\le (1+\sigma(\sn))^{1/2}\varepsilon^{1/2}\cdot\sqrt2\sum_{i=1}^2\|\Omega_i\|_{L\log L(\sn)}^{1/2}.
\end{align*}
Then inequality (\ref{eq4.9}) follows easily from the fact that
$$\Big\|\prod_{i=1}^2(|\Omega_{i,\varepsilon}|-\varepsilon)\Big\|_{L^{1/2}{(\sn\backslash\widetilde S_\varepsilon)}}^{1/2} \le \|4\varepsilon^2\|_{L^{1/2}(\sn\backslash\widetilde S_\varepsilon)}^{1/2} \le 2\sigma(\sn)\varepsilon.$$

Since $\varepsilon$ is arbitrary and small enough, we get
\begin{equation}\label{eq4.10}
	\varliminf_{\lambda\to0^+}\lambda^{1/2}|G_\lambda| \ge \frac{\|\Omega_1\Omega_2\|_{L^{1/2}(\sn)}^{1/2}}n.
\end{equation}

Finally, from (\ref{eq4.7}) and (\ref{eq4.10}), we obtain that
$$\lim_{\lambda\to0^+}\lambda^{1/2}|G_\lambda|=\frac{\|\Omega_1\Omega_2\|_{L^{1/2}(\sn)}^{1/2}}n.$$
\end{proof}

\subsection{Proof of Theorem \ref{thm1} (ii)}\label{s42}

\begin{proof} For $\lambda>0$, we set
\begin{align*}
	H_\lambda =& \bigg\{x:\Big||T_{\Vec\Omega}(f_1,f_2)(x)| - \frac{|\Omega_1(\frac x{|x|})\Omega_2(\frac x{|x|})|}{|x|^{2n}}\Big|>\lambda\bigg\}; \\
	H^1_\lambda =& \bigg\{|x|>\widetilde R_\varepsilon:\Big||T_{\vec\Omega_\varepsilon}(g_1,g_2)(x)|-\frac{|\Omega_{1,\varepsilon}(\frac x{|x|})\Omega_{2,\varepsilon}(\frac x{|x|})|}{|x|^{2n}}\Big|>\lambda\bigg\}; \\
	H_\lambda^2 =& \bigg\{x:\Big|\frac{|\Omega_1(\frac x{|x|})\Omega_2(\frac x{|x|})|-|\Omega_{1,\varepsilon}(\frac x{|x|})\Omega_{2,\varepsilon}(\frac x{|x|})|}{|x|^{2n}}\Big|>\lambda\bigg\}.
\end{align*}
Now we claim that
\begin{equation}\label{eq4.11}
	H_\lambda \subset H^1_{(1-2\sqrt\varepsilon)\lambda} \cup H_{\sqrt\varepsilon\lambda}^2 \cup \bigcup_{i=2,3} (G_{\sqrt{\varepsilon}\lambda/2}^{1,i} \cup G_{\sqrt{\varepsilon}\lambda/2}^i) \cup \overline{B(0,\widetilde R_\varepsilon)}.
\end{equation}

To prove (\ref{eq4.11}), it is sufficient to show that the complementary set of right side is contained in $F_\lambda^c$. For any $x\in {H^1_{(1-2\sqrt\varepsilon)\lambda}}^c \cap {H^2_{\sqrt\varepsilon\lambda}}^c \cap {G_{\sqrt{\varepsilon}\lambda/2}^{1,2}}^c \cap {G_{\sqrt{\varepsilon}\lambda/2}^2}^c \cap \overline{B(0,\widetilde R_\varepsilon)}^c$, it holds that $|x|>R_\varepsilon$ and
$$\Big|\frac{|\Omega_1(\frac x{|x|})\Omega_2(\frac x{|x|})|-|\Omega_{1,\varepsilon}(\frac x{|x|})\Omega_{2,\varepsilon}(\frac x{|x|})|}{|x|^{2n}}\Big| \le \sqrt\varepsilon\lambda;$$
$$\Big||T_{\vec\Omega_\varepsilon}(g_1,g_2)(x)|-\frac{|\Omega_{1,\varepsilon}(\frac x{|x|})\Omega_{2,\varepsilon}(\frac x{|x|})|}{|x|^{2n}}\Big| \le (1-2\sqrt\varepsilon)\lambda;$$
$$|T_{\Vec\Omega}(g_1,h_2)(x)| + |T_{\Vec\Omega}(h_1,g_2)(x)| + |T_{\Vec\Omega}(h_1,h_2)(x)| \le \frac{\sqrt\varepsilon}2\lambda;$$
$$|T_{(\Omega_{1,\varepsilon},\Omega_2-\Omega_{2,\varepsilon})}(g_1,g_2)(x)| + |T_{(\Omega_1-\Omega_{1,\varepsilon},\Omega_{2,\varepsilon})}(g_1,g_2)(x)| + |T_{(\Omega_1-\Omega_{1,\varepsilon},\Omega_2-\Omega_{2,\varepsilon})}(g_1,g_2)(x)| \le \frac{\sqrt\varepsilon}2\lambda.$$
From these inequalities, it's easy to see that
\begin{align*}
	|T_{\Vec\Omega}(f_1,f_2)(x)| &\le |T_{\vec\Omega}(g_1,g_2)(x)| + \frac{\sqrt\varepsilon}2\lambda \le |T_{\vec\Omega_\varepsilon}(g_1,g_2)(x)| + \sqrt\varepsilon\lambda \\
	&\le \frac{|\Omega_{1,\varepsilon}(\frac x{|x|})\Omega_{2,\varepsilon}(\frac x{|x|})|}{|x|^{2n}} + (1-\sqrt\varepsilon)\lambda \le \frac{|\Omega_1(\frac x{|x|})\Omega_2(\frac x{|x|})|}{|x|^{2n}} + \lambda.
\end{align*}

For $x\in {H^1_{(1-2\sqrt\varepsilon)\lambda}}^c \cap {H^2_{\sqrt\varepsilon\lambda}}^c \cap {G_{\sqrt{\varepsilon}\lambda/2}^{1,3}}^c \cap {G_{\sqrt{\varepsilon}\lambda/2}^3}^c \cap \overline{B(0,\widetilde R_\varepsilon)}^c$, we have
\begin{align*}
	|T_{\Vec\Omega}(f_1,f_2)(x)| &\ge |T_{\Vec\Omega}(g_1,g_2)(x)| - \frac{\sqrt\varepsilon}2\lambda \ge |T_{\vec\Omega_\varepsilon}(g_1,g_2)(x)| - \sqrt\varepsilon\lambda \\
	&\ge \frac{|\Omega_{1,\varepsilon}(\frac x{|x|})\Omega_{2,\varepsilon}(\frac x{|x|})|}{|x|^{2n}} - (1-\sqrt\varepsilon)\lambda \ge \frac{|\Omega_1(\frac x{|x|})\Omega_2(\frac x{|x|})|}{|x|^{2n}} - \lambda.
\end{align*}
These two estimates imply that
$$\Big||T_{\Vec\Omega}(f_1,f_2)(x)| - \frac{|\Omega_1(\frac x{|x|})\Omega_2(\frac x{|x|})|}{|x|^{2n}}\Big| \le \lambda,$$
which means that $x\in H_\lambda^c$. Then the proof of claim (\ref{eq4.11}) is finished.

It remains to estimate $|H^1_{(1-2\sqrt\varepsilon)\lambda}|$ and $|H_{\sqrt\varepsilon\lambda}^2|$. We split the proof into three parts.

\vspace{0.3cm}

\noindent\textbf{Part 1: Upper estimate for $|H_{\sqrt\varepsilon\lambda}^2|$.}

Applying Lemma \ref{lem21} and using the method as in Part 2 of Subsection \ref{s41}, we conclude that
\begin{equation}\label{eq4.12}
	\begin{aligned}
		|H_{\sqrt\varepsilon\lambda}^2| &= \frac{\||\Omega_1\Omega_2| - |\Omega_{1,\varepsilon}\Omega_{2,\varepsilon}|\|_{L^{1/2}(\sn)}^{1/2}}{n\varepsilon^{1/4}\lambda^{1/2}} \\
		&\le C\sum\limits_{i=1}^2\|\Omega_i\|_{L\log L(\sn)}^{1/2}\frac{\varepsilon^{1/4}}{\lambda^{1/2}} \\
		&\le C\sum\limits_{i=1}^2(\|\Omega_i\|_{L\log L(\sn)}+1)^{1/2}\frac{\varepsilon^{1/8}}{\lambda^{1/2}}.
	\end{aligned}
\end{equation}

\vspace{0.3cm}

\noindent\textbf{Part 2: Upper estimate for $|H_{(1-2\sqrt\varepsilon)\lambda}^1 \cap \widetilde V_\varepsilon|$.}

For $x\in H_{(1-2\sqrt\varepsilon)\lambda}^1 \cap \widetilde V_\varepsilon$, by (\ref{eq4.5}) and (\ref{eq4.8}), we know that
\begin{align*}
	\frac{(1-\varepsilon)(|\Omega_{1,\varepsilon}(\frac x{|x|})|-\varepsilon)(|\Omega_{2,\varepsilon}(\frac x{|x|})|-\varepsilon)}{(1+\sqrt\varepsilon)^{2n}|x|^{2n}} \le& |T_{\vec\Omega_\varepsilon}(g_1,g_2)(x)| \\
	\le& \frac{(|\Omega_{1,\varepsilon}(\frac x{|x|})|+\varepsilon)(|\Omega_{2,\varepsilon}(\frac x{|x|})|+\varepsilon)}{(1-\sqrt\varepsilon)^{2n}|x|^{2n}}
\end{align*}
and
\begin{align*}
	\frac{(1-\varepsilon)(|\Omega_{1,\varepsilon}(\frac x{|x|})|-\varepsilon)(|\Omega_{2,\varepsilon}(\frac x{|x|})|-\varepsilon)}{(1+\sqrt\varepsilon)^{2n}|x|^{2n}} &\le \frac{|\Omega_{1,\varepsilon}(\frac x{|x|})\Omega_{2,\varepsilon}(\frac x{|x|})|}{|x|^{2n}} \\
	&\le \frac{(|\Omega_{1,\varepsilon}(\frac x{|x|})|+\varepsilon)(|\Omega_{2,\varepsilon}(\frac x{|x|})|+\varepsilon)}{(1-\sqrt\varepsilon)^{2n}|x|^{2n}}.
\end{align*}
These controls imply that
\begin{align*}
	\Big||T_{\vec\Omega_\varepsilon}(g_1,g_2)(x)| - \frac{|\Omega_{1,\varepsilon}(\frac x{|x|})\Omega_{2,\varepsilon}(\frac x{|x|})|}{|x|^{2n}}\Big| \le& \frac{(|\Omega_{1,\varepsilon}(\frac x{|x|})|+\varepsilon)(|\Omega_{2,\varepsilon}(\frac x{|x|})|+\varepsilon)}{(1-\sqrt\varepsilon)^{2n}|x|^{2n}} \\
	&- \frac{(1-\varepsilon)(|\Omega_{1,\varepsilon}(\frac x{|x|})|-\varepsilon)(|\Omega_{2,\varepsilon}(\frac x{|x|})|-\varepsilon)}{(1+\sqrt\varepsilon)^{2n}|x|^{2n}} \\
	=&: \frac{\widetilde I_{\vec\Omega,\varepsilon}(\frac x{|x|})}{|x|^{2n}}.
\end{align*}
Thus Lemma \ref{lem21} gives that
\begin{align*}
	|H_{(1-2\sqrt\varepsilon)\lambda}^1 \cap V_\varepsilon| \le \Big|\Big\{x\in V_\varepsilon:\frac{\widetilde I_{\vec\Omega,\varepsilon}(\frac x{|x|})}{|x|^{2n}}>(1-2\sqrt\varepsilon)\lambda\Big\}\Big| = \frac{\|\widetilde I_{\vec\Omega,\varepsilon}\|_{L^{1/2}(S_\varepsilon)}^{1/2}}{n(1-2\sqrt\varepsilon)^{1/2}\lambda^{1/2}}.
\end{align*}

The same argument as in Part 2 of Subsection \ref{s32} implies that
$$\|\widetilde I_{\vec\Omega,\varepsilon}\|_{L^{1/2}(S_\varepsilon)}^{1/2} \le C\prod_{i=1}^2(\|\Omega_i\|_{L\log L(\sn)}+1)^{1/2}\varepsilon^{1/8}$$
Then
\begin{equation}\label{eq4.13}
	|H_{(1-2\sqrt\varepsilon)\lambda}^1 \cap V_\varepsilon| \le C\prod_{i=1}^2(\|\Omega_i\|_{L\log L(\sn)}+1)^{1/2}\frac{\varepsilon^{1/8}}{\lambda^{1/2}}.
\end{equation}

\vspace{0.3cm}

\noindent{\bf Part 3: Upper estimate for $|H_{(1-2\sqrt\varepsilon)\lambda}^1 \cap \widetilde V_\varepsilon^c|$.}

For all $|x|>\widetilde R_\varepsilon$, it follows from the inequality (\ref{eq4.5}) that
$$\Big||T_{\vec\Omega_\varepsilon}(g_1,g_2)(x)| - \frac{|\Omega_{1,\varepsilon}(\frac x{|x|})\Omega_{2,\varepsilon}(\frac x{|x|})|}{|x|^{2n}}\Big| \le 2\frac{(|\Omega_{1,\varepsilon}(\frac x{|x|})|+\varepsilon)(|\Omega_{2,\varepsilon}(\frac x{|x|})|+\varepsilon)}{(1-2\sqrt\varepsilon)^{2n}|x|^{2n}}.$$
Then by Lemma \ref{lem21}, $|H_{(1-2\sqrt\varepsilon)\lambda}^1\cap \widetilde V_\varepsilon^c|$ is dominated by
$$|H_{(1-2\sqrt\varepsilon)\lambda}^1\cap \widetilde V_\varepsilon^c| \le \frac{\sqrt2\Big\|\prod\limits_{i=1}^2(|\Omega_{i,\varepsilon}|+\varepsilon)\Big\|_{L^{1/2}(\sn\backslash S_\varepsilon)}^{1/2}}{n(1-2\sqrt\varepsilon)^{n+1/2}\lambda^{1/2}}.$$

Let $\theta\in\sn\backslash S_\varepsilon$, then at least one of $|\Omega_{1,\varepsilon}(\theta)|$ and $|\Omega_{2,\varepsilon}(\theta)|$ should be not more than $\varepsilon$. So we let
$$\widetilde S_\varepsilon' = \{\theta\in\sn:|\Omega_{1,\varepsilon}(\theta)| \le \varepsilon\}\quad\text{ and }\quad \widetilde S_\varepsilon'' = \{\theta\in\sn:|\Omega_{2,\varepsilon}(\theta)| \le \varepsilon\}.$$
It is obvious that $\sn\backslash \widetilde S_\varepsilon \subset \widetilde S_\varepsilon' \cup \widetilde S_\varepsilon''$, and
\begin{align*}
	\Big\|\prod\limits_{i=1}^2(|\Omega_{i,\varepsilon}|+\varepsilon)\Big\|_{L^{1/2}(\sn\backslash \widetilde S_\varepsilon)}^{1/2} &\le \Big\|\prod\limits_{i=1}^2(|\Omega_{i,\varepsilon}|+\varepsilon)\Big\|_{L^{1/2}(\widetilde S_\varepsilon')}^{1/2} + \Big\|\prod\limits_{i=1}^2(|\Omega_{i,\varepsilon}|+\varepsilon)\Big\|_{L^{1/2}(\widetilde S_\varepsilon'')}^{1/2} \\
	&\le \sum_{i=1}^2\left(\|2\varepsilon\|_{L^1(\sn)}\||\Omega_{i,\varepsilon}|+\varepsilon\|_{L^1(\sn)}\right)^{1/2} \\
	&\le C\sum_{i=1}^2(\|\Omega_i\|_{L\log L(\sn)}+\sigma(\sn)\varepsilon)^{1/2}\varepsilon^{1/2} \\
	&\le C\sum_{i=1}^2(\|\Omega_i\|_{L\log L(\sn)}+1)^{1/2}\varepsilon^{1/8}.
\end{align*}
Then we have the upper estimate for $|F_{(1-2\sqrt\varepsilon)\lambda}^1 \cap V_\varepsilon^c|$:
\begin{equation}\label{eq4.14}
	|H_{(1-2\sqrt\varepsilon)\lambda}^1 \cap \widetilde V_\varepsilon^c| \le C\sum\limits_{i=1}^2(|\Omega_i\|_{L\log L(\sn)}+1)^{1/2}\frac{\varepsilon^{1/8}}{\lambda^{1/2}}.
\end{equation}

Finally, by (\ref{eq4.2}), (\ref{eq4.3}), (\ref{eq4.11}), (\ref{eq4.12}), (\ref{eq4.13}) and (\ref{eq4.14}), it holds that
\begin{align*}
	|H_\lambda| \le C\Big(\prod_{i=1}^2(\|\Omega_{i,\varepsilon}\|_{L\log L(\sn)}+1)^{1/2} + \sum_{i=1}^2(\|\Omega_{i,\varepsilon}\|_{L\log L(\sn)}+1)^{1/2}\Big)\frac{\varepsilon^{1/8}}{\lambda^{1/2}} + |\overline{B(0,\widetilde R_\varepsilon)}|.
\end{align*}
Multiplying $\lambda^{1/2}$ on both sides, and by the arbitrariness of $\varepsilon$, we finally get
$$\lim_{\lambda\to0^+}\lambda^{1/2}|H_\lambda|=0.$$
This finishes the proof of Theorem \ref{thm2} (ii) and also completes the proof of Theorem \ref{thm2}.
\end{proof}

\bigskip

\noindent Moyan Qin

\smallskip

\noindent {\it Address:} Laboratory of Mathematics and Complex Systems
(Ministry of Education of China),
School of Mathematical Sciences, Beijing Normal University,
Beijing 100875, People's Republic of China

\smallskip

\noindent {\it E-mail:} \texttt{myqin@mail.bnu.edu.cn}

\medskip

\noindent Huoxiong Wu

\smallskip

\noindent {\it Address:} School of Mathematical Sciences, Xiamen University, Xiamen 361005, China

\smallskip

\noindent {\it E-mail:} \texttt{huoxwu@xmu.edu.cn}

\medskip

\noindent Qingying Xue

\smallskip

\noindent {\it Address:} Laboratory of Mathematics and Complex Systems
(Ministry of Education of China),
School of Mathematical Sciences, Beijing Normal University,
Beijing 100875, People's Republic of China

\smallskip

\noindent {\it E-mail:} \texttt{qyxue@bnu.edu.cn}

\end{document}